\documentclass{article}
\usepackage{amsmath,amssymb,amsthm}

\newtheorem{thm}{Theorem}

\newtheorem{lem}[thm]{Lemma}
\newtheorem{claim}{Claim}

\theoremstyle{plain}

\newcommand{\leftexp}[2]{{\vphantom{#2}}^{#1}{#2}}
\newcommand{\suz}{\leftexp{2}{B_2}}
\newcommand{\ree}{\leftexp{2}{G_2}}
\newcommand{\su}{\leftexp{2}{A_2}}
\newcommand{\fr}{\textrm{Fr}}
\newcommand{\frobenius}{\textrm{Frobenius}}
\newcommand{\FR}{\mathfrak{F}}
\newcommand{\FE}{\textrm{FR}}
\newcommand{\Sp}{\textrm{Sp}_4}
\newcommand{\SpS}{\textrm{SymSp}_4}
\newcommand{\Sps}{\SpS}

\newcommand{\sps}{\SpS}
\newcommand{\F}{\mathbb{F}}
\newcommand{\pair}[2]{\left( #1 , #2 \right)}
\newcommand{\Aut}{\textrm{Aut}}
\newcommand{\Oc}{\mathbb{O}}
\newcommand{\C}{\mathcal{C}}
\newcommand{\V}{\mathcal{V}^3}
\newcommand{\ad}{\textrm{ad}}
\newcommand{\Nm}{\textrm{Nm}}
\newcommand{\End}{\textrm{End}}
\newcommand{\der}{\textrm{Der}}
\newcommand{\inder}{\textrm{InDer}}

\title{Canonical Projective Embeddings of the Deligne-Lusztig Curves Associated to $\su$, $\suz$ and $\ree$}
\author{Daniel M. Kane}

\begin{document}

\maketitle

\begin{abstract}
The Deligne-Lusztig varieties associated to the Coxeter classes of the algebraic groups $\su,\suz$ and $\ree$ are affine algebraic curves.  We produce explicit projective models of the closures of these curves.  Furthermore for $d$ the Coxeter number of these groups, we find polynomials for each of these models that cut out the $\F_q$-points, the $\F_{q^d}$-points and the $\F_{q^{d+1}}$-points, and demonstrate a relation satisfied by these polynomials.
\end{abstract}

\section{Introduction}

There are four (twisted) Chevalley groups of rank 1.  The associated Deligne-Lusztig varieties for the Coxeter classes of these groups all give affine algebraic curves.  The completions of these curves have several applications including the representation theory of the associated group (\cite{DL,Lusz}), coding theory (\cite{Codes}) and the construction of potentially interesting covers of $\mathbb{P}^1$ (\cite{Gross}).

In this paper, we consider these curves associated to the groups $G=\su,\suz$ and $\ree$.  The remaining curve is associated to $G=A_1$ and is $\mathbb{P}^1$, but we do not cover this case as it is easy and doesn't follow many of the patterns found in the analysis of the other three cases.  For each of these curves, we explicitly construct an embedding $C\hookrightarrow \mathbb{P}(W)$ where $W$ is a representation of $G$ of dimension 3,5, or 14 respectively, and provide an explicit system of equations cutting out $C$.  The curve associated to $\su$ is the Fermat curve.  The curve associated to $\suz$ is also well-known though not immediately isomorphic to our embedding.  As far as I know there is no standard embedding of the curve associated to $\ree$, although in \cite{ReeCurve} they construct an explicit curve with the correct genus, symmetry group and number points, which is likely the correct curve.

For each case, let $d$ be the Coxeter number of the associated group, namely $d=3,4,6$ for $\su,\suz,\ree$, respectively.  In \cite{Gross}, Gross proved that for each of these curves $C$, that $C/G^\sigma\cong \mathbb{P}^1$, with the corresponding map $C\rightarrow \mathbb{P}^1$ ramified over only three points corresponding to the images of the $\F_q$-points, the $\F_{q^d}$-points and the $\F_{q^{d+1}}$-points.  In all cases, we write this map explicitly by finding the homogenous polynomials on $C$ that correspond to the pullbacks of the degree-1 functions on $\mathbb{P}^1$ vanishing at each of these points, and demonstrating a linear relation between these functions.

In all cases we attempt to make our constructions canonical.  We define the algebraic group $G$ as the group of automorphisms of some vector space $V$ preserving some additional structure.  We define the Frobenius map on $G$ by picking an isomorphism between $V$ and some other space $V'$ constructed functorially from $V$ (for example, for $\su$, we get the Frobenius map defining $SU(3)$ by picking a Frobenius-linear map between $V$ and its dual).  From $V$ we construct $W$, another representation of $G$, given as a quotient of $\Lambda^2 V$.  In each case we define our map $C\rightarrow \mathbb{P}(W)$ by sending a Borel, $B$, of $G$ to the 2-form corresponding to the line in $V$ that $B$ fixes.  The construction of the functions giving our map from $C$ to $\mathbb{P}^1$ are all given by linear algebraic constructions.

There are a number of similarities in our techniques for the three different cases, suggesting that there may be a more general way to deal with all three at once, although we were unable to find such a technique.  In addition to the similarity of overall approach, much of the feel of these constructions should be the same although they differ in the details.  Additionally in all three cases we compute the degree of the embedding and find that it is given by $\frac{|G^\sigma|}{|B^\sigma||T^\sigma|}$ (here $T$ is a Coxeter torus of $G$).

In Section 2, we describe some of the basic theory along with an outline of our general approach.  In Section 3, we deal with the case of $\su$; in Section 4, the case of $\suz$; and in Section 5, we deal with $\ree$.  Much of the exposition in Section 2 is somewhat abstract and corresponds to relatively simple computations in Section 3, so if you are having trouble following in Section 2, it is suggested that you look at Section 3 in parallel to get a concrete example of what is going on.

\section{Preliminaries}

Here we provide an overview of the techniques and notation that will be common to our treatment of all three cases.  In Section 2.1, we review the definition and basic theory of Deligne-Lusztig curves.  In Section 2.2, we discuss some representations of $G$ which will prove useful in our later constructions.  In Section 2.3, we give a more complete overview of the common techniques to our different cases and describe a general category-theoretic construction that will provide us with the necessary Frobenius maps in each case.  Finally in Section 2.4, we fix a couple of points of notation for the rest of the paper.

\subsection{Basic Theory of the Deligne-Lusztig Curve}

Let $G$ be an algebraic group of type $A_2,B_2$ or $G_2$, defined over a finite field $\F_q$ with $q$ equal to $q_0^2,2q_0^2$ or $3q_0^2$ respectively.
Let $\FE$ be the $\F_q$-Frobenius of $G$, and let $\sigma$ be a Frobenius map so that $\FE=\sigma^2.$  We pick a an element $w$ of the Weyl group of $G$ of height 1.  The Deligne-Lusztig variety is then defined to be the subvariety of the flag variety of $G$ consisting of the Borel subgroups $B$ so that $B$ and $\sigma(B)$ are in relative position either $w$ or 1 (actually it is usually defined to be just the $B$ where $B$ and $\sigma(B)$ are in position $w$, but we use this definition, which constructs the completed curve).  The resulting variety has an obvious $G^\sigma$ action and in all of these cases is a smooth complete algebraic curve.

In each of these cases, let $B$ a Borel subgroup of $G$.  Let $T$ be a twisted Coxeter torus of $G$, that is a $\sigma$-invariant maximal torus such that the action of $\sigma$ on $T$ is conjugate to the action of $w\sigma$ on a split torus.  Let $d$ be $3,4,6$ for $A_2,B_2,G_2$ respectively.  We will later make use of several facts about the points on the Deligne-Lusztig curve, $C$, defined over various fields and their behavior under the action of $G^\sigma$.

Here and throughout the rest of the paper we will use the phrase $\F_{q^n}$-point to mean a point defined over $\F_{q^n}$ but not defined over any smaller extension of $\F_q$.  We make use of the following theorem of Lusztig:
\begin{thm}
\begin{enumerate}
\item $G^\sigma$ acts transitively on the $\F_q$-points of $C$ with stabilizer $B^\sigma$.
\item $C$ has no $\F_{q^n}$-points for $1<n<d$.
\item $G^\sigma$ acts transitively on the $\F_{q^d}$-points of $C$ with stabilizer $T^\sigma$.
\item $G^\sigma$ acts simply transitively on the $\F_{q^{d+1}}$-points of $C$.
\item $G^\sigma$ acts freely on the $\F_{q^n}$-points of $C$ for $n>d+1$.
\end{enumerate}
\end{thm}

Recall that $C/G^\sigma\cong \mathbb{P}^1$ with the points of ramification given by the three points corresponding to the orbit of $\F_q$-points, the orbit of $\F_{q^d}$-points, and the orbit of $\F_{q^{d+1}}$-points.

We also make use of some basic counts (given in \cite{Gross}).  In particular, for $A_2$:
\begin{align*}
d & = 3\\
|G^\sigma| & = q_0^3(q_0^3+1)(q-1)\\
|B^\sigma| & = q_0^3(q-1)\\
|T^\sigma| & = q - q_0 + 1\\
\#\{\F_q-\textrm{points of} \ C\} & = q_0^3+1\\
\#\{\F_{q^3}-\textrm{points of} \ C\} & = q_0^3(q_0+1)(q-1)\\
\#\{\F_{q^4}-\textrm{points of} \ C\} & = q_0^3(q_0^3+1)(q-1).
\end{align*}
For $B_2$:
\begin{align*}
d & = 4\\
|G^\sigma| & = q^2(q^2+1)(q-1)\\
|B^\sigma| & = q^2(q-1)\\
|T^\sigma| & = q - 2q_0 + 1\\
\#\{\F_q-\textrm{points of} \ C\} & = q^2+1\\
\#\{\F_{q^4}-\textrm{points of} \ C\} & = q^2(q+2q_0+1)(q-1)\\
\#\{\F_{q^5}-\textrm{points of} \ C\} & = q^2(q^2+1)(q-1).
\end{align*}
For $G_2$:
\begin{align*}
d & = 6\\
|G^\sigma| & = q^3(q^3+1)(q-1)\\
|B^\sigma| & = q^3(q-1)\\
|T^\sigma| & = q - 3q_0 + 1\\
\#\{\F_q-\textrm{points of} \ C\} & = q^3+1\\
\#\{\F_{q^6}-\textrm{points of} \ C\} & = q^3(q+3q_0+1)(q^2-1)\\
\#\{\F_{q^7}-\textrm{points of} \ C\} & = q^3(q^3+1)(q-1).
\end{align*}

\subsection{Representation Theory}

Let $G$ be as above.  Fix a Borel subgroup $B$.  $B$ is contained in two maximal parabolic subgroups, $P_1$ and $P_2$, corresponding to the short root and the long root respectively.  There exists a representation $V$ of $G$ so that $B$ fixes the complete flag $0\subset L \subset M \subset S \subset \ldots \subset V$ and so that $P_1$ is the subgroup of $G$ fixing $L$, and $P_2$ the subgroup fixing $M$.  If $G$ is $A_2,B_2$ or $G_2$, the dimension of $V$ is $3,4$ or $7$ respectively.

Inside of $\Lambda^2 V$ is the representation $W$ of $G$ given by $W\subset \Lambda^2 V$ is the subrepresentation containing $\Lambda^2 M$.  In our three cases, $\Lambda^2 V/W$ equals $0,1$ or $V$ respectively.  If we are in any characteristic for $A_2$, characteristic 2 for $B_2$, or characteristic 3 for $G_2$, $W$ has a quotient representation, $V'$ of the same dimension as $V$ and so that the image of $G\rightarrow \End(V')$ is isomorphic to $G$.  Picking an isomorphism between $G$ and its image provides an endomorphism of $G$.  If $G=A_2$, this is its outer automorphism.  For $G$ equal to $B_2$ or $G_2$ in characteristic 2 or 3 respectively, this endomorphism squares to the Frobenius endomorphism over the relevant prime field, giving a definition of $\sigma$.

\subsection{Basic Techniques}

Our basic techniques will be similar in all three cases and are as follows.  For $p$ a prime, let $q_0=p^m$, and $q=p^{2m}$ or $q=p^{2m+1}$ as appropriate.  Let $\C$ be the groupoid where an object of $\C$ is a representation of an algebraic group abstractly isomorphic to the representation $V$ of $G$, and a morphism of $\C$ is an isomorphism of representations.  In each case, we will reinterpret $\C$ as a groupoid whose objects are merely vector spaces with some additional structure (for example in the case of $\su$ an object of $\C$ will be a three dimensional vector space with a volume form).  This reinterpretation will allow us to work more concretely with $\C$ in each individual case, but as these structures are specific to which case we are in, we will ignore them for now.

In each case there will be two functors of interest from $\C$ to itself.  The first is the Frobenius functor, $\fr:\C\rightarrow\C$. This functor should be thought of as abstractly applying $\F_p$-Frobenius to each element.  In particular there should be a natural $\frobenius$-linear transformation from $V$ to $\fr(V)$.  In particular each of our $\C$ will have objects that are vector spaces, perhaps with some extra structure.  As a vector space we will define $\fr(V)$ as $\{[v]:v\in V\}$ where addition and multiplication are defined by $[v]+[w]=[v+w]$ and $k[v]=[\frobenius^{-1}(k)v]$. The morphisms are unchanged by $\fr$.  This should all be compatible with the extra structure accorded to an object of $\C$, except for inner products which must also be twisted by Frobenius. Note that giving a morphism $T:V\rightarrow \fr^n(W)$ is the same as giving a $\frobenius^{-n}$-linear map $S:V \rightarrow W$ that respects the additional $\C$-structure. Note therefore that a morphism $T:V\rightarrow \fr^{n}(V)$ gives $V$ an $\F_{p^n}$-structure.  By abuse of notation, we will also use $\fr$ to denote that natural $\frobenius$-linear map $V\rightarrow \fr(V)$ defined by $v\rightarrow [v]$.

The other functor of importance is $':\C\rightarrow\C$.  This is the functor that takes $V$ and gives $V'$ as a subquotient of $\Lambda^2 V$. The exact construction of $'$ will vary from case to case, for $\su$ $V'$ will be given by the dual of $V$.  It will be clear in all cases that there is a natural equivalence between $\fr\circ '$ and $'\circ \fr$.

Define $a$ to be $0$ or $1$ so that $q=p^{2m+a}$.  In each case, we will also find a natural transformation $\rho$ from $\fr^a$ to $''$ (again the details vary by case and we will not go into them here, though for $\su$ it is the obvious isomorphism between a vector space and the dual of its dual).  Lastly we pick a $V\in \C$ and a morphism $F:V'\rightarrow \fr^{-m}(V)$.

This is enough to give a more explicit definition of $\sigma$.  We construct a Frobenius map for $G=\Aut(V)$.  We note the morphisms $\fr^{-a}(\rho_V):V\rightarrow \fr^{-a}(V'')$, $\fr^{-a}(F'):\fr^{-a}(V'')\rightarrow \fr^{-m-a}(V')$, and $\fr^{-m-a}(F):\fr^{-m-a}(V')\rightarrow \fr^{-2m-a}(V)$.  The composition
$$
\FR := \fr^{-m-a}(F) \circ (\fr^{-a}(F')) \circ \fr^{-a}(\rho_V): V \rightarrow \fr^{-2m-a}(V)
$$
defines an $\F_q$-structure on $V$.

This gives us a Frobenius endomorphism $\FE:G\rightarrow G$ defined by $\fr^{-2m-a}(\FE(g))\circ \FR = \FR \circ g$.  We define $\sigma:G\rightarrow G$ by $\fr^{-m}(\sigma(g)) \circ F = F \circ g'$.  We have that:
\begin{align*}
\sigma(\sigma(g))
& = \fr^{m}(F) \circ \fr^{m}(\sigma(g')) \circ \fr^{m}(F^{-1})\\
& = \fr^{m}(F) \circ \fr^{2m}(F') \circ \fr^{2m}(g'') \circ \fr^{2m}(F'^{-1}) \circ \fr^{m}(F^{-1})\\
& = \fr^{2m+a}(\FR\circ \fr^{-a}(\rho_V^{-1}) \circ \fr^{-a}(g'') \circ \fr^{-a}(\rho_V) \circ \FR^{-1}) \\
& = \fr^{2m+a}(\FR\circ \fr^{-a}(\rho_V^{-1} \circ g'' \circ \rho_V) \circ \FR^{-1}) \\
& = \fr^{2m+a}(\FR\circ \fr^{-a}(\fr^{a}(g))\circ \FR^{-1})\\
& = \fr^{2m+a}(\FR\circ g \circ \FR^{-1})\\
& = \FE(g).
\end{align*}
The third to last step above comes from the fact that $\rho$ is a natural transformation. This gives us an endomorphism $\sigma:G\rightarrow G$ so that $\sigma^2=\FE$.

We note that this technique for defining $\sigma$ works most conveniently when by ``algebraic group'' we mean ``group object in the category of varieties over $\overline{\F_q}$'', since then $G$ can be associated with $\textrm{Aut}_{\C}(V)$, and we have an action of $\sigma$ on $G$.  On the other hand if you want ``algebraic group'' to mean ``group object in the category of schemes'', then the same technique should still work as long as we consider $\C$ as a category enriched in schemes.

We may pick our element $w$ so that two Borels of $G$ are in relative position $w$ if they fix the same line in $V$.  We then define a projective embedding $C\hookrightarrow \mathbb{P}(W)$ sending a Borel $B$ to the two-form defined by the plane it fixes in $V$. In each case we will provide explicit polynomials that cut out the image of $C$.  We will compute the degree of this embedding by finding a polynomial that vanishes exactly at the $\F_q$-points of $C$.  In each case this degree will be $\frac{|G^\sigma|}{|B^\sigma||T^\sigma|}$.  In each case for each $\F_q$-point, there is a hyperplane that intersects $C$ only at that point but with large multiplicity.  We also compute polynomials that cut out the $\F_{q^d}$-points and the $\F_{q^{d+1}}$-points.  Lastly we find a linear relation between appropriate powers of these polynomials.

\subsection{Notes}

Throughout this paper by an $\F_{q^n}$-point of a curve defined over $\F_q$, we will mean a point defined over $\F_{q^n}$ but not over any smaller extension of $\F_q$.  Also throughout this paper a $\frobenius^n$-linear map will always refer to the $n^{th}$ power of the Frobenius map over the corresponding prime field.

\section{The Curve Associated to $^{2}{A_2}$}

\subsection{The Group $A_2$ and its Representations}

One of the groups associated to the Lie Algebra $A_2$ is the group $G=\textrm{SL}_3$.  This group acts naturally on a three dimensional vector space $V$.  The Borels of $G$ are defined by picking an arbitrary flag $0\subset L \subset M \subset V$.  As we range over Borel subgroups $\Lambda^2 M$ will span all of $\Lambda^2 V$, so our representation $W$ will be given by $W=\Lambda^2 V$.

\subsection{A Canonical Definition of $^2A_2$}

The group $A_2$ is just $\textrm{SL}_3$.  The group $\su$ will turn out to be simply the special unitary group.  Although this would be simple to derive directly, we will attempt to use the same basic technique as we will for the more complicated groups.  Let $\V$ be the groupoid consisting of all three dimensional vector spaces over $\F:=\overline{F_q}$ with a volume form, $\Omega$.  We define $\fr:\V\rightarrow\V$ as above.  We define $':\V\rightarrow\V$ by letting $V'$ equal $\textrm{Hom}(V,\F)$, the dual of $V$.  Note that $V'$ is naturally $\Lambda^2 V$ with the pairing $\pair{v}{\omega}=\frac{v\wedge\omega}{\Omega}$.  We use these definitions interchangeably.

We have the obvious natural transformation $\rho:\textrm{Id}\Rightarrow ''$ so that $\rho_V(v)$ is the functional $\phi\in V' \rightarrow \phi(v)$.  Given $V\in \V$ and $F:V'\rightarrow \fr^{-m}(V)$, we can define $\sigma$ by $\fr^{-m}(\sigma(g)) \circ F = F \circ g'$.  It is not hard to see that $F$ defines a hermitian inner product on $V$ and that $\sigma(g)$ is simply the adjoint of $g$ with respect to this hermitian form.

\subsection{The Deligne-Lusztig Curve}

Let $B$ be the Borel fixing the line $L=\langle v\rangle$ and the plane $M=\langle v,w\rangle = \ker(\phi)$, where $v\in V,\phi\in V'$.  Note that we may think of $\phi$ as $\omega := v\wedge w\in\Lambda^2 V.$  Now for $g\in B$, since $g$ fixes $M$, $g'$ must fix the line in $V'$ containing $\omega$.  Similarly, since $g$ fixes $L$, $g'$ must fix the plane $L\wedge V$ in $V'$.  Hence if $B$ is defined by $L=\langle v\rangle$ and $M=\langle v,w\rangle$ and if $u$ is some other linearly independent vector, then $\sigma(B)$ is defined by $\langle F(v\wedge w)\rangle$ and $\langle F(v\wedge w), F(v\wedge u)\rangle$.

Now for $B$ to correspond to a point on the Deligne-Lusztig curve, it must therefore be the case that $L=\langle F(v\wedge w)\rangle$.  We define the embedding of the Deligne-Lusztig curve $C$ to $\mathbb{P}(\Lambda^2 V)$ by sending $B$ to the line containing $\omega = v\wedge w$.  We note that this is an embedding since given $\omega$, we know that $\langle v\rangle = \langle F(\omega)\rangle$.  This is a smooth embedding since the coordinates of $C$ can be written as polynomials on $\mathbb{P}(\Lambda^2 V)$.  For simplicity of notation we denote the pairing between $V$ and $V'$ as $\pair{-}{-}$.  We note that the image of $C$ is cut out by the equation:
$$
\pair{\omega}{F(\omega)} = 0.
$$
Thinking of $F$ as defining a hermitian form on $V'$, this just says that the norm of $\omega$ with respect to this hermitian form is 0.  Hence $C$ is just the Fermat curve of degree $q_0+1$.

\subsection{Divisors of Note}

We first compute a divisor that vanishes exactly on the $\F_q$-points.  We note that a Borel $B$ corresponds to a point over $\F_q$ if and only if $\sigma(\sigma(B))=B$.  We claim that this holds if and only if $\sigma(B)=B$.  One direction holds trivially.  For the other direction, if $\sigma(\sigma(B))=B$, then $\sigma(\sigma(B))$ and fixes the same line as $B$ and $\sigma(B)$.  But the line fixed by $\sigma(B)$ is determined by the plane fixed by $B$. Hence $B$ and $\sigma(B)$ fix both the same line and the same plane and hence are equal.  Hence $B$ corresponds to an $\F_q$-point if and only if $F(\langle v\wedge w, v\wedge u\rangle) = \langle v,w\rangle$.  Since $F(v\wedge w) \propto v$, this happens exactly when $F(v\wedge u) \in \langle v,w\rangle$.

Letting $u$ be some vector not in $M$ we consider
\begin{equation}\label{suPoly}
\frac{(F(F(\omega)\wedge u)\wedge \omega)\otimes \Omega^{\otimes (q_0-1)}}{(\omega \wedge u)^{\otimes q_0}}.
\end{equation}
Note that both numerator and denominator are multiples of $\Omega^{\otimes q_0}$ so the fraction makes sense.  Note also that numerator and denominator are homogeneous of degree $q_0$ in $u$ as an element of $V/M$.  Hence the resulting expression is independent of $u$ and hence a polynomial of degree $q-q_0+1$ in $\omega.$  It should be noted that this polynomial vanishes exactly when $(F(F(\omega)\wedge u)\wedge \omega)\propto F(v\wedge u)\wedge v \wedge w$ does, or in other words exactly at the $\F_q$-points.  We could show that it vanishes simply at these points merely be a computation of degrees (knowing that we have the Fermat curve), but instead we will show it directly as that will be useful in our later examples, thus giving us another way of computing the degree of $C$.

Considering an analytic neighborhood of an $\F_q$-point in $C$, we may compute the polynomial in Equation \ref{suPoly} with $u$ held constant.  We would like to show that the derivative is non-zero.  This is clearly equivalent to showing that the derivative of $(F(F(\omega)\wedge u)\wedge \omega)$ is non-zero.  Suppose for sake of contradiction that it were zero.  It is clear that $F(\omega)\wedge \omega$ is identically 0.  Note also that the derivatives of $F(F(\omega)\wedge u)$ and $F(\omega)$ are both 0.  Hence this would mean that $(F(F(\omega)\wedge u)\wedge d\omega = F(\omega)\wedge d\omega = 0$.  But since $F(\omega)$ and $F(F(\omega)\wedge u)$ span $M$, this can only happen if $d\omega$ is proportional to $\omega$, which means $d\omega=0$ since we are in a projective space.  Hence the polynomial in \ref{suPoly} vanishes exactly at the $\F_q$-points of $C$ and with multiplicity 1.  This proves that $C$ is of degree $q_0+1$.

Recall that the polynomial $(\omega,F(\omega))$ is identically 0 on $C$.  Consider the polynomial $\pair{\omega}{F(\FR(\omega))}$.  This polynomial is clearly $G^\sigma$ invariant.  Since the divisor it defines has degree $(q_0+1)(q_0^3+1)$, it cannot vanish on any $G^\sigma$ orbit other than the orbit of $\F_q$-points.  Furthermore near an $\F_q$-point, $\omega_0$, this divisor agrees with $\pair{\omega}{F(\FR(\omega_0))}$ to order $q_0^3$.  Since the former vanishes to order at most $q_0+1$, $\pair{\omega}{F(\FR(\omega))}$ cannot vanish identically, and hence defines the divisor that vanishes to degree $q_0+1$ on each of the $\F_q$-points.  Note also that since this divisor agrees with $\pair{\omega}{F(\FR(\omega_0))}$ to order $q_0^3$, the linear divisor $\pair{\omega}{F(\FR(\omega_0))}$ vanishes only at $\omega=\omega_0$ and to degree $q_0+1$ and nowhere else.

Next consider the divisor $f=\pair{\omega}{F(\FR^2(\omega))}$. Note that if $\omega$ is defined over $\F_{q^3}$ that $\FR(f) = \pair{\FR(\omega)}{F(\omega)} = \pair{F(\omega)}{\omega} = 0.$  Hence the divisor defined by $f$ vanishes on the $\F_{q^3}$-points of $C$.  Additionally $f$ agrees with $\pair{\omega}{F(\FR(\omega))}$ to order $q_0^3$ on any $\F_q$-point.  Therefore $f$ vanishes to order 1 on the $\F_{q^3}$-points, and order $q_0+1$ on the $\F_q$-points.  This accounts for the entirety of the degree of the divisor so it therefore vanishes nowhere else.

Finally consider the divisor defined by $\pair{\omega}{F(\FR^3(\omega))}$.  Similarly to above this must vanish on the $\F_{q^4}$-points of $C$ and to order exactly $q_0+1$ on the $\F_q$-points.  The remaining degree unaccounted for is
$$
(q_0^7+1)(q_0+1) - (q_0^3+1)(q_0+1) - q_0^3(q_0^3+1)(q_0-1)(q_0+1) = q_0^3(q_0+1)(q_0^3-q_0).
$$
Since the remainder of the divisor is $G^\sigma$-invariant and cannot vanish on $\F_q$-points, the only orbit small enough is that of the $\F_{q^3}$-points.  Hence the remainder of the divisor must be $q_0$ times the sum of the $\F_{q^3}$-points.  Hence $\pair{\omega}{F(\FR^3(\omega))}$ vanishes to degree $q_0+1$ on the $\F_q$-points, to order $q_0$ on the $\F_{q^3}$-points, to order 1 on the $\F_{q^4}$-points, and nowhere else.

We now introduce several important polynomials.  Let $F_3 = \pair{\omega}{F(\FR(\omega))}, F_5 = \pair{\omega}{F(\FR^2(\omega))}, F_7 = \pair{\omega}{F(\FR^3(\omega))}.$  Let $P_1 = F_3^{1/(q_0+1)}$.  We know that such a root exists since the $(q_0+1)^{st}$ root of the divisor of $F_3$ is the sum of the $\F_q$-points of $C$, and we know of polynomials with that divisor.  Hence $P_1$ is a polynomial whose divisor is the sum of the $\F_q$-points of $C$.  Let $P_3 = \frac{F_5}{F_3}$ and let $P_4 = \frac{F_7}{F_3P_3^{q_0}}$.  We know that these are polynomials (instead of just rational functions) because their associated divisors are the sum of the $\F_{q^3}$-points and the sum of the $\F_{q^4}$-points respectively.

We claim that the following relation holds:
\begin{equation}\label{surelation}
P_4-P_3^{q-q_0+1}+P_1^{q_0^3(q_0-1)}=0.
\end{equation}
Although the proof of Equation \ref{surelation} is somewhat complicated we can much more easily prove that
\begin{equation}\label{surelmodified}
P_4-P_3^{q-q_0+1}+aP_1^{q_0^3(q-1)}=0
\end{equation}
for some $(q_0+1)^{st}$ root of unity $a$.  We do this by showing that for $a$ properly chosen the above vanishes on the $\F_{q^4}$-points of $C$ and the $\F_q$-points of $C$.  Since this is more points than allowed by the degree of the polynomial, it implies that $P_4-P_3^{q-q_0+1}+aP_1^{q_0^3(q_0-1)}$ must vanish identically.  To show that it vanishes on the $\F_{q^4}$-points we show that for $Q\in C$ an $\F_{q^4}$-point that
$$
P_3^{q_0^3+1}(Q) = P_1^{q_0^3(q-1)(q_0+1)}(Q)=F_3^{q_0^5-q_0^3}(Q).
$$
This is obviously equivalent to showing that
$$
F_5^{q_0^3+1}(Q) = F_3^{q_0^5+1}(Q).
$$
If $Q$ corresponds to a vector $\omega\in V$ defined over $\F_{q^4}$, we have that
\begin{align*}
F_5^{q_0^3+1}(Q)
& = \pair{\omega}{F(\FR^2(\omega))}\cdot\pair{F(\FR(\omega))}{\FR^4(\omega)} \\
& = \pair{\omega}{F(\FR^2(\omega))}\cdot\pair{\omega}{F(\FR(\omega))} \\
& = \pair{\omega}{F(\FR(\omega))}\cdot\pair{F(\FR^2(\omega))}{\FR^4(\omega)} \\
& = F_3^{q_0^5+1}(Q). \\
\end{align*}

Next we consider the $\F_q$-point.  We need to show that if $Q\in C(\F_q)$ that
$$
P_4(Q) = P_3^{q-q_0+1}(Q).
$$
Equivalently we will show that
$$
F_7F_3^{q} - F_5^{q+1}
$$
vanishes to degree more than $(q+1)(q_0+1)$ at $Q$.  This is easy since if we use our parameter $\omega$ with $\omega=\omega_0$ at $Q$ (where $\omega_0=\FR(\omega_0)$) than up to order $q_0^3+q(q_0+1)$ the above is
$$
\pair{\omega}{F(\omega_0)}\cdot\pair{\omega}{F(\omega_0)}^q - \pair{\omega}{F(\omega_0)}^{q+1} = 0.
$$

This completes our proof of Equation \ref{surelmodified}.

\section{The Curve Associated to $^{2}{B_2}$}

\subsection{The Group $B_2$ and its Representations}

A form of $B_2$ is $\textrm{Sp}_4$.  It has a natural representation on a four-dimensional symplectic vector space $V$.  The Borels of $G$ correspond to complete flags $0\subset L \subset M \subset L^\perp \subset V$ where $M$ is a lagrangian plane.  We have another representation $W\subset \Lambda^2 V$ given as the kernel of the map $\Lambda^2 V \rightarrow \F$ defined by the alternating form on $V$.  Equivalently, $W$ is the subset consisting the 2-forms $\alpha$ so that $\alpha\wedge \omega=0$, where $\omega$ is the 2-form corresponding to the alternating form on $V$.  In characteristic 2, $\omega$ is contained in $W$, so we can take the quotient $V' = W/\langle \omega \rangle$.

\subsection{A Canonical Definition of $^2B_2$}

We define the groupoid $\SpS$ whose objects are symplectic spaces of dimension 4 over $\F=\overline{\F_2}$, and whose morphisms are symplectic linear transformations.  We will write objects of $\Sps$ either as $V$ or as $(V,\omega)$, where $V$ is underlying vector space and $\omega$ is the 2-form associated to the symplectic form.

To each object $(V,\omega)$ in $\SpS$, we can pick a symplectic basis $e_0,e_1,f_0,f_1$ of $V$ so that $\pair{e_i}{e_j} = \pair{f_i}{f_j}=0, \pair{e_i}{f_j} = \delta_{i,j}$.  We next note that there is a canonical volume form $\mu\in \Lambda^4 V$.  In any characteristic other than 2, we could use $\mu=\omega\wedge \omega$, but here $\omega\wedge \omega=0$.  We instead use $\mu=e_0\wedge e_1\wedge f_0 \wedge f_1$.  We have to prove:
\begin{lem}
$\mu=e_0\wedge e_1\wedge f_0 \wedge f_1$ is independent of the choice of symplectic basis.
\end{lem}
\begin{proof}
Any two symplectic bases can be interchanged by some symplectic transformation $A:V\rightarrow V$.  We need only show that $\det(A)=1$.  Writing $A$ in the basis of our first symplectic basis and letting $$J=\left[ \begin{matrix} 0 & 1 & 0 & 0 \\ 1 & 0 & 0 & 0 \\ 0 & 0 & 0 & 1 \\ 0 & 0 & 1 & 0 \end{matrix}\right],$$the statement that $A$ is symplectic becomes $J=AJA^{T}$.  Therefore taking determinants we find $\det(A)^2=1$ so $\det(A)=1$ (since we are in characteristic 2).
\end{proof}

Next we note that for $(V,\omega)\in \sps$, that we have a bilinear form on $\Lambda^2 V$ given by $\pair{\alpha}{\beta} = \frac{\alpha\wedge\beta}{\mu}$.  We note that this pairing is alternating since if $e_I$ is a wedge of two 1-forms then clearly $\pair{e_I}{e_I}=0$ and if $\alpha = \sum_I c_I e_I$ then
$$\pair{\alpha}{\alpha}=\sum_{I,J} c_Ic_J\pair{e_I}{e_J} = \sum_I c_I^2 \pair{e_I}{e_I} + 2\sum_{I<J} c_Ic_J \pair{e_I}{e_J} = 0.$$
This alternating form is inherited by the subquotient $V'=\omega^\perp/\langle \omega \rangle$ of $\Lambda^2 V$. Lastly we note that the pairing on $V'$ is non-degenerate. This can be seen by picking a symplectic basis, $e_i,f_i$, and noting that $e_0\wedge e_1, e_0 \wedge f_1, f_0\wedge f_1, e_1\wedge f_1$ is a symplectic basis for $V'$.  This allows us to define a functor $':\sps\rightarrow\sps$ by $V\rightarrow V'$ as above.

We also have the functor $\fr:\sps\rightarrow\sps$ as in Section 2.

We next define a natural transformation:
$$
\rho:''\Rightarrow \fr.
$$
To do this we must produce a $\frobenius^{-1}$-linear map $\rho_V:V''\rightarrow V$ for $V\in \sps$.  We provide this by producing a (linear,$\frobenius$-linear), pairing $V''\times V\rightarrow \F$.  We define this pairing on $(\alpha\wedge \beta,v)$ (for $\alpha,\beta \in \Lambda^2 V, v\in V$) by
\begin{equation}\label{rhooneequ}
(\alpha\wedge \beta,v) = \pair{i_v(\alpha)}{i_v(\beta)}.
\end{equation}
We note that the function in Equation \ref{rhooneequ} thought of as a map from $\Lambda^2 V\times \Lambda^2 V\times V\rightarrow \F$ is clearly linear and anti-symmetric in $\alpha$ and $\beta$, and clearly homogeneous of degree 2 in $v$.  In order to show that it is well defined on $V''\times V$ we need to demonstrate that it vanishes if either $\alpha$ or $\beta$ is a multiple of $\omega$, that the map on $\Lambda^2 V'\times V$ vanishes on the symplectic form of $V'$, and that it is additive in $v$.  The first of these is true since
$$
\pair{i_v(\omega)}{i_v(\alpha)} = \pair{v}{i_v(\alpha)} = \alpha(v*,v*) = 0.
$$
Where $v*$ is the dual of $v$, and the last equation holds since $\alpha$ is alternating.  Let $e_i,f_i$ be a symplectic basis of $V$.

Note that we have a basis
$$
\{(e_0\wedge e_1)\wedge(e_0\wedge f_1),(e_0\wedge e_1)\wedge(e_1\wedge f_0),(e_1\wedge f_0)\wedge (f_0\wedge f_1),(e_0\wedge f_1)\wedge(f_0\wedge f_1)\}
$$
of $V''$. To prove that our function is additive in $v$, notice that it suffices to check that
$$
\pair{i_u(\alpha)}{i_v(\beta)} = \pair{i_v(\alpha)}{i_u(\beta)}
$$
for all $u,v\in V$.    Hence it suffices to check the above for $\alpha= a\wedge b, \beta=a\wedge c$ for $a,b,c\in V$ and $\pair{a}{b}=\pair{a}{c}=0$.  In this case
\begin{align*}
\pair{i_u(\alpha)}{i_v(\beta)} & = \pair{i_u(a\wedge b)}{i_v(a\wedge c)} \\
& = \pair{\pair{u}{a}b+\pair{u}{b} a}{\pair{v}{a}c+\pair{v}{c}a} \\
& = \pair{a}{u}\pair{a}{v}\pair{b}{c}.
\end{align*}
Since this is symmetric in $u$ and $v$, the operator is additive.

Hence we have a $(\textrm{linear},\frobenius-\textrm{linear})$ function $\Lambda^2 V'\times V\rightarrow \F$.  To show that it descends to a function $V''\times V\rightarrow \F$, it suffices to check that our form vanishes on
$$
(e_0\wedge e_1)\wedge(f_0\wedge f_1) + (e_0\wedge f_1)\wedge(f_0\wedge e_1),
$$
which is clear from checking on basis vectors (in fact our form vanishes on each term in the above sum) and by additivity.

Finally we need to show that $\rho_V$ is symplectic.  This can be done by picking a symplectic basis and noting that by the above:
\begin{align*}
\rho_V((e_0\wedge e_1)\wedge(e_0\wedge f_1)) &= e_0\\
\rho_V((e_0\wedge e_1)\wedge(e_1\wedge f_0)) &= e_1\\
\rho_V((e_1\wedge f_0)\wedge(f_0\wedge f_1)) &= f_0\\
\rho_V((e_0\wedge f_1)\wedge(f_0\wedge f_1)) &= f_1.
\end{align*}
It is clear that $\rho$ defines a natural transformation.

Now given an $F:V'\rightarrow \fr^{-m}(V)$, we get an $\F_q$-Frobenius $\FR$ on $V$, and we can now define $\sigma:G\rightarrow G$ as in Section 2 so that $\sigma^2=\FE$.

\subsection{The Deligne-Lusztig Curve}

For simplicity from here on out we will think of $F$ as a $\frobenius^{m}$-linear map from $V'\rightarrow V$ rather than a linear one from $V'\rightarrow \fr^{-m}(V)$.

We recall that for $V\in \sps$ that $G=\Sp(V)$ has Borel subgroups that correspond to the data of a line $L\subset V$ and a Lagrangian subspace $M\subset V$ so that $L\subset M$.  The corresponding Borel, $B$ is the set of elements $g$ that fix both $L$ and $M$.  The Deligne-Lusztig curve can be thought of as the set of Borel subgroups, $B$, in the flag variety so that $B$ and $\sigma(B)$ fix the same line.

Consider the Borel subgroup $B$ fixing the line $\langle v \rangle$ and the plane $\langle v,w\rangle$ with $\pair{v}{w}=0$.  Notice that since $\pair{v}{w}=0$ that $\omega \wedge v \wedge w = 0$, and thus $v \wedge w$ can be thought of as a (necessarily non-zero) element of $V'$.  If $g$ fixes $\langle v,w\rangle$, then $g'$ clearly fixes the line of $v\wedge w$.  Hence $F(v\wedge w)$ must be the line fixed by $\sigma(B)$.

Given $V$ and $F$ as above, we define the Deligne-Lusztig curve $C$ and produce an embedding $C\hookrightarrow \mathbb{P}(W)$ by sending $B$ to the line containing $\alpha=v \wedge w$ (recall $W\subset \Lambda^2 V$ was the orthogonal compliment of $\omega$).  This is an embedding since if we pick a point in the image we have fixed both the plane fixed by $B$ and the line fixed by $\sigma(B)$ (and hence also the line fixed by $B$ since they are the same).  This embedding is smooth since all of the other coordinates of $B,\sigma(B)$ can be written as polynomials in $\alpha$.  Furthermore this embedding is given by a few simple equations.  Namely:
\begin{itemize}
\item The quadratic relation equivalent to $\alpha$ being a pure wedge of two vectors (i.e. that it lies on the Grassmannian).
\item The span of $\alpha, F(\alpha)\wedge V$ must have dimension at most 3.
\end{itemize}
The first condition guarantees that $F(\alpha)\neq 0$.  The second condition implies that if we set $v=F(\alpha)$ that $\alpha = [v\wedge w]$ for some $w$.  The fact that $\alpha\in W$ implies that $\pair{v}{w}=0$.  Together these imply that we have a point of the image.

These equations also cut out the curve scheme-theoretically.  We show this for $m>0$ by showing that the tangent space to the scheme defined by these equations is one dimensional at every $\bar{\F}$-point.  If we are at some point $\alpha$, we note that the derivative of $F(\alpha)$ is 0, so along any tangent line, $d\alpha$ must lie in $F(\alpha)\wedge V$.  This restricts $d\alpha$ to a three-dimensional subspace of $\Lambda^2 V$.  Additionally, $\pair{\omega}{d\alpha}$ must be zero, restricting to two dimensions.  Finally we are reduced to one dimension when we project onto $\Lambda^2 V /\langle \alpha \rangle$, which is the tangent space to projective space at $\alpha$.

\subsection{Divisors of Note}

Let $q=2^{2m+1}$.  Let $q_0=2^m$.  We wish to compute the degree of the above embedding of $C$.  We do this by producing a polynomial that exactly cuts out the $F_q$-points of $C$ and thus must be a divisor of degree $q^2+1$.  The $\F_q$-points are the points corresponding to Borel subgroups $B$ so that $B=\sigma(\sigma(B))$.  First we claim that these are exactly the Borels for which $B=\sigma(B)$.  This is because $B$ and $\sigma(B)$ automatically fix the same line $\langle v \rangle$.  If $\sigma(B)$ and $\sigma(\sigma(B))$ fix the same line then the 2-forms corresponding to the planes fixed by $B$ and $\sigma(B)$ must by multiples of each-other modulo $\omega$.  But since $\omega$ cannot be written as a pure wedge of $v$ with any vector, this implies that $B$ and $\sigma(B)$ must fix the same plane, and hence be equal.

For any Borel $B$ in $C$, we can pick $v,w,u\in V$ so that $B$ fixes the line $L=\langle v \rangle$, the Lagrangian plane $M=\langle v,w\rangle$ and the space $S=\langle v \rangle^\perp = \langle v,w,u \rangle$.  Notice that for $g\in B$ that $g'$ fixes the Lagrangian plane $\langle v\wedge w, v\wedge u\rangle$.  Hence $\sigma(B)$ fixes the plane $\langle F(v\wedge w),F(v\wedge u)\rangle$.  Since $F(v\wedge w)$ is parallel to $v$, $B=\sigma(B)$ if and only if $F(v\wedge u)$ lies in $M$. Note that since $v\wedge w\perp v\wedge u$ in $V'$ that $F(v\wedge u)$ must lie in $S$.  We can then consider:
\begin{equation}\label{suzptsexp}
\left(\frac{\alpha}{v\wedge w} \right)^{q-2q_0+1}\left(\frac{F(v\wedge u)/M}{u/M} \right)\pair{u}{w}^{1-q_0}\left(\frac{F(v\wedge w)}{v} \right)^{2q_0-1}.
\end{equation}
Note that the value in Expression \ref{suzptsexp} is independent of our choice of $v,w,u$, is never infinite, is zero exactly when $B=\sigma(B)$, and is homogenous of degree $q-2q_0+1$ in $\alpha$.  Therefore it defines a polynomial of degree $q-2q_0+1$ in our embedding that vanishes exactly on the $F_q$-points.  We have left to show that it vanishes to order 1 at these points.  Consider picking a local coordinate around an $F_q$-point of $C$ and computing the derivative of Expression \ref{suzptsexp} with respect to this local coordinate.  Clearly this derivative is some non-zero multiple of the derivative of
$$
\left(\frac{F(v\wedge u)/M}{u/M} \right).
$$
This can be rewritten as
$$
\frac{F(v\wedge u)\wedge \alpha}{u\wedge \alpha}.
$$

So in addition to picking a local coordinate, $z$, for $\alpha$ we can pick a local coordinate for $v,u$ as well.  We can let $v=F(\alpha)$, and let $u$ be a fixed vector perpendicular to both $v$ and $dv/dz$.  Then for $m>0$ we have that $dF(v\wedge u)/dz=0$.  Note that $\alpha$ must always be a pure wedge of $F(\alpha)$ with some perpendicular vector.  Since $dF(\alpha)=0$, $d\alpha$ must lie in $\langle v\wedge w, v\wedge u\rangle$.  It cannot be parallel to $v\wedge w$ though since this is parallel to $\alpha$.  But, since $F(v \wedge u)$ is in $\langle v,w\rangle$ and not parallel to $v$, this means that $F(v\wedge u)\wedge d\alpha$ is non-zero.  Hence the derivative of Expression \ref{suzptsexp} with respect to $z$ at an $\F_q$-point is non-zero, and hence Expression \ref{suzptsexp} defines a divisor that is exactly the sum of the $\F_q$-points of $C$.  Therefore the degree of our embedding must be:
$$
\frac{q^2+1}{q-2q_0+1} = q+2q_0+1.
$$

Note that the divisor $\pair{\alpha}{\alpha}$ is identically 0, where the pairing is as elements of $V'$.

Next consider the polynomial $\pair{\alpha}{\FR(\alpha)}$.  We claim that this function is identically zero on $C$.  We prove this by contradiction.  If it were non-zero it would define a divisor of degree $(q+1)(q+2q_0+1)$ that would be clearly invariant under the action of $G^\sigma$ on $C$.  On the other hand, the only orbit small enough to be covered by this is the orbit of the $\F_q$-points, whose number does not divide the necessary degree.

Next consider the polynomial $\pair{\alpha}{\FR^2(\alpha)}$.  We claim that the divisor of this polynomial is simply $q+2q_0+1$ times the sum of the $\F_q$-points of $C$. Again the divisor must be $G^\sigma$ invariant, and again the only orbit of small enough order is the orbit of $\F_q$-points.  Hence if the polynomial is not uniquely zero it must be the divisor specified.  On the other hand, consider a local coordinate $\alpha(z)$ around an $\F_q$-point.  Then this divisor is equal to $\pair{\alpha}{\alpha_0} + O(z^{q^2})$ ($\alpha_0=\alpha(0)$).  $\pair{\alpha}{\alpha_0}$  cannot vanish to degree more than $q+2q_0+1$ since it is a degree 1 polynomial.  Hence $\pair{\alpha}{\FR^2(\alpha)}$ vanishes to degree exactly $q+2q_0+1$ on each $\F_q$-point and nowhere else.  Furthermore if $\beta$ is the coordinate of an $\F_q$-point, then the degree 1 polynomial $\pair{\alpha}{\beta}$ vanishes at this point to degree $q+2q_0+1$ and nowhere else.

Next consider the polynomial $\pair{\alpha}{\FR^3(\alpha)}$.  This corresponds to a divisor of degree $(q^3+1)(q+2q_0+1)$.  Note that if $\beta$ is the coordinate vector for an $\F_{q^4}$-point of $C$ that $\pair{\beta}{\FR^{-1}(\beta)}=0$. Since $\beta=\FR^4(\beta)$, this implies that $\pair{\beta}{\FR^3(\beta)}=0$.  Therefore this divisor contains each of the $\F_{q^4}$ points.  Also by the above this divisor vanishes to degree exactly $q+2q_0+1$ on the $\F_q$-points (and is hence non-zero).  We have just accounted for a total degree of
$$
q^2(q+2q_0+1)(q-1) + (q^2+1)(q+2q_0+1) = (q+2q_0+1)(q^3+1),
$$
thus accounting for all of the vanishing of the polynomial.  Hence this polynomial defines the divisor given by the sum of the $\F_{q^4}$-points plus $(q+2q_0+1)$ times the sum of the $\F_q$-points.

Lastly, consider the divisor $\pair{\alpha}{\FR^4(\alpha)}$.  By the arguments above it vanishes on the $\F_{q^5}$-points and to degree exactly $q+2q_0+1$ on the $\F_q$-points. We have so far accounted for a divisor of total degree
$$
q^2(q^2+1)(q-1) + (q+2q_0+1)(q^2+1) = (q+2q_0+1)(q^4-2q_0q^3+2q_0q^2+1).
$$
We are missing a divisor of degree $(q+2q_0+1)q^2(q-1)2q_0.$  This divisor must be $G^\sigma$-invariant and cannot contain the orbit of $\F_q$-points.  Hence the only orbit small enough is that of the $\F_{q^4}$-points, which can be taken $2q_0$ times.  Hence this polynomial defines the divisor equal to the sum of the $\F_{q^5}$-points plus $2q_0$ times the sum of the $\F_{q^4}$-points, plus $q+2q_0+1$ times the sum of the $\F_q$-points.

Note that above we demonstrated that there was a polynomial that vanished to degree 1 exactly on the $\F_q$-points.  Call this polynomial $P_1$.  Note that we can choose $P_1$ so that $P_1^{q+2q_0+1}=\pair{\alpha}{\FR^2(\alpha)}$.  We also have polynomials $P_4=\frac{\pair{\alpha}{\FR^3(\alpha)}}{\pair{\alpha}{\FR^2(\alpha)}}$ that vanishes to degree 1 on the $\F_{q^4}$ points and nowhere else.  Finally we have $P_5 = \frac{\pair{\alpha}{\FR^4(\alpha)}}{P_4^{2q_0}\pair{\alpha}{\FR^2(\alpha)}}$ that vanishes exactly on the $\F_{q^5}$-points.  We claim that if the correct root $P_1$ is chosen that:
$$
P_1^{q^2(q-1)}+P_4^{q-2q_0+1}+P_5=0.
$$
We do this by demonstrating that this polynomial vanishes on all of the $\F_{q^5}$-points and all of the $\F_q$-points, which is more than a polynomial of its degree should be able to.  In fact since everything can actually be defined over $\F_2$, the correct choice of $P_1$ will have to be the one given in Expression \ref{suzptsexp}.

We begin by showing vanishing on the $\F_{q^5}$-points.  We first show that $\pair{\alpha}{\FR^2(\alpha)}^{q^2(q-1)}=P_4^{q^2+1}$ on the $\F_{q^5}$-points.  We then note that by taking $P_1$ to be an appropriate root of $\pair{\alpha}{\FR^2(\alpha)}$ we can cause $P_1^{q^2(q-1)}+P_4^{q-2q_0+1}$ to vanish on some $\F_{q^5}$-point.  But then we note that $P_1$ is a multiple of the quantity in Expression \ref{suzptsexp}, which is clearly $G^\sigma$ invariant, and hence $P_1^{q^2(q-1)}+P_4^{q-2q_0+1}$ must vanish on all $\F_{q^5}$-points.

Let $\alpha$ be an $\F_{q^5}$-point.  We wish to show that $\pair{\alpha}{\FR^2(\alpha)}^{q^2(q-1)}=P_4^{q^2+1}(\alpha)$.  This is equivalent to asking that
$$
\pair{\alpha}{\FR^2(\alpha)}^{q^3+1} = \pair{\alpha}{\FR^3(\alpha)}^{q^2+1}.
$$
But this is clear since
\begin{align*}
\pair{\alpha}{\FR^2(\alpha)}^{q^3+1}
& = \pair{\alpha}{\FR^2(\alpha)}^{q^3}\pair{\alpha}{\FR^2(\alpha)}\\
& = \pair{\FR^3(\alpha)}{\FR^5(\alpha)}\pair{\alpha}{\FR^2(\alpha)}\\
& = \pair{\alpha}{\FR^3(\alpha)}\pair{\FR^5(\alpha)}{\FR^2(\alpha)}\\
& = \pair{\alpha}{\FR^3(\alpha)}\pair{\alpha}{\FR^3(\alpha)}^{q^2}\\
& = \pair{\alpha}{\FR^3(\alpha)}^{q^2+1}.
\end{align*}

Next we need to show that at an $\F_q$-point that $P_4(\alpha)^{q-2q_0+1}=P_5(\alpha)$.  This is equivalent to showing that
$$
\left(\frac{\pair{\alpha}{\FR^3(\alpha)}}{\pair{\alpha}{\FR^2(\alpha)}} \right)^{q+1} = \frac{\pair{\alpha}{\FR^4(\alpha)}}{\pair{\alpha}{\FR^2(\alpha)}}.
$$
Equivalently, we will show that
$$
\pair{\alpha}{\FR^3(\alpha)}^{q+1} + \pair{\alpha}{\FR^4(\alpha)}\pair{\alpha}{\FR^2(\alpha)}^q
$$
vanishes to degree more than $(q+1)(q+2q_0+1)$ on an $\F_q$-point.  The above is equal to
$$
\pair{\FR(\alpha)}{\FR^4(\alpha)}\pair{\alpha}{\FR^3(\alpha)} + \pair{\alpha}{\FR^4(\alpha)}\pair{\FR(\alpha)}{\FR^3(\alpha)}.
$$
Setting $\alpha=\beta+z$, where $\beta$ is defined over $\F_q$,
\begin{align*}
& \pair{\FR(\alpha)}{\FR^4(\alpha)}\pair{\alpha}{\FR^3(\alpha)} + \pair{\alpha}{\FR^4(\alpha)}\pair{\FR(\alpha)}{\FR^3(\alpha)}\\
= & \pair{\beta+z^q}{\beta}\pair{\beta+z}{\beta} + \pair{\beta+z}{\beta}\pair{\beta+z^q}{\beta} + O(z^{q^3}) \\
= & \pair{z^q}{\beta}\pair{z}{\beta} + \pair{z}{\beta}\pair{z^q}{\beta} + O(z^{q^3})\\
= & O(z^{q^3}).
\end{align*}
This completes the proof of our identity.

\section{The Curve Associated to $^2G_2$}

\subsection{Description of $^2G_2$ and its Representations}

We begin by describing the group $G_2$.  $G_2$ can be thought of as the group of automorphisms of an octonian algebra.  This non-associative algebra can be given by generators $a_i$ for $i\in \mathbb{Z}/7$ where the multiplication for $a_i,a_{i+1},a_{i+3}$ is given by the relations for the quaternions for $i,j,k$.  The octonians have a natural conjugation which sends $a_i$ to $-a_i$ and 1 to 1.  Instead of thinking of the automorphisms of the entire octonian algebra, we will think about the automorphisms of the space of pure imaginary octonians (which must be preserved by any automorphism since they are the non-central elements whose squares are central).  There are two pieces of structure on the pure imaginary octonians that allow one to reconstruct the full algebra. Firstly, there is a non-degenerate, symmetric pairing $\pair{x}{y}$ given by the real part of $x\cdot y$.  Also there is an anti-symmetric, bilinear operator $x*y$ which is the imaginary part of $x\cdot y$.  Lastly there is an anti-symmetric cubic form given by $(x,y,z)$ is the real part of $x\cdot(y\cdot z)$, or $\pair{x}{y*z}$.

These pieces of structure are not unrelated.  In particular, since any two octonians $x,y$ satisfy $x(xy) = (xx)y$ we have for $x$ and $y$ pure imaginary that
$$
\pair{x}{x} y = x*(x*y) + x\pair{x}{y}
$$
or
$$
x*(x*y) = \pair{x}{x} y - \pair{x}{y}x.
$$
Therefore $\pair{x}{x}$ is the eigenvalue of $y\rightarrow x*(x*y)$ on a dimension 6 subspace.  Hence $\pair{-}{-}$ is determined by $*$.

Hence $G$ has a 7 dimensional representation on $V$, the pure imaginary octonians.  The Borels of $G$ fix a flag $0\subset L \subset M \subset S \subset S^\perp \subset M^\perp \subset L^\perp \subset V$.  Where here we have that $\pair{M}{M}=0$, $M*M=0$, $S=L^\perp*L$.  Also $\Lambda^2 V$ has a subrepresentation $W$ of dimension 14 given by the kernel of $*:\Lambda^2 V \rightarrow V$.  As we shall see, in characteristic 3, $W \supset W^\perp$ thus giving a 7 dimensional representation $V' = W/W^\perp$.

\subsection{Definition of $\sigma$}

A property of note is that in characteristic 3, $*$ makes the pure imaginary octonians into a Lie Algebra.  The Jacobi identity is easily checked on generators.  We now let $\F=\overline{\F_3}$.  We define $\Oc$ to be the groupoid of Lie Algebras isomorphic to the pure imaginary octonians over $\F$ with the operation $*$.  We claim that for $V\in\Oc$ that the Lie Algebra of outer derivations of $V$ is another Lie Algebra in $\Oc$ (this is not too hard to check using the standard basis).  This obviously defines a functor $':\Oc\rightarrow\Oc$.

We will discuss a little bit of the structure of these outer derivations.  First note that since a derivation is a differential automorphism, and since an automorphism must preserve $\pair{-}{-}$, that these derivations must be alternating linear operators.  This means that we can think of the derivations of $V$ as lying inside of $\Lambda^2 V$.  It is not hard to check that $a_0\wedge a_1 + a_2\wedge a_5$ defines a derivation.  From symmetry and the fact that the dimension of the space of derivations is the dimension of $G_2$, which is 14, we find that the space of all derivations must be $W=\ker(*:\Lambda^2 V \rightarrow V)$.  It is not hard to see that the space of inner derivations is equal to $W^\perp\subset W$.  Hence $V'=W/W^\perp$ is isomorphic to the space of outer derivations.

We will need to know something about the structure of the operator $\ad(x):V\rightarrow V$ when $x\in V$ has $\pair{x}{x}=0$ in characteristic 3.  Since $\ad(x)(\ad(x)(y)) = x\pair{x}{y}$ we know that $\ad(x)^2$ has a kernel of dimension 6.  Therefore $\ad(x)$ has a kernel of dimension at least 3.  We claim that the dimension of this kernel is exactly 3.  Suppose that $x*y=0$.  By the above this implies that $\pair{x}{y}=0$.  Hence, as octonians, it must be the case that $x\cdot y=0$.  This means that if $\Nm$ is the multiplicative norm on the octonians and if we take lifts $\overline{x},\overline{y}$ of $x$ and $y$ to the rational octonians, that $\Nm(\overline{x}\cdot \overline{y})$ is a multiple of 9.  Since $\Nm(\overline{x})$ cannot be a multiple of 9 this means that $\Nm(\overline{y})$ is a multiple of 3.  Hence $\pair{y}{y}=0$.  Hence $\ker(\ad(x))$ is a null-plane for $\pair{-}{-}$, and hence cannot have dimension more than 3.  Notice that this implies that $\ker(\ad(x)) = \ad(x)(\langle x\rangle^\perp).$

As before we have a functor $\fr:\Oc\rightarrow \Oc$.

Once again our definition of $\ree$ will depend upon finding a natural transformation between $''$ and $\fr$.  This time we will find a natural equivalence in the other direction (since the inverse will be easier to write down).  We will find a natural transformation $\rho:\fr\Rightarrow ''$.  This amounts to finding a Frobenius-linear isomorphism from $V$ to $V''$.  The basic idea will be to think of derivations of $V$ as differential automorphisms.  In particular, given a derivation $d$, we can find an element $1+\epsilon d + O(\epsilon^2)$ of $\End(V)[[\epsilon]]$ that is an automorphism of $V$.

Let $x\in V$ and let $d$ be a derivation of $V$.  We let $e_x$ be an automorphism of the form $1+\epsilon \ad(x) + \epsilon^2/2 \ad(x)^2+ O(\epsilon^3)$, and let $e_d$ be an automorphism of the form $1+\delta d + O(\delta^2)$.  We claim that the commutator of $e_x$ and $e_d$ is $1+(\textrm{inner derivations})+\epsilon^3 \delta r +O(\epsilon^4 \delta)$ where the $O$ assumes that $\delta\ll \epsilon$ and where $r$ is some derivation of $V$.  We claim that this defines a map $(\rho_V(x))(d)=r$ so that $\rho_V(x)$ gives a well-defined outer automorphism of $V'$, and that $\rho_V:V\rightarrow V''$ is Frobenius-linear.  We prove this in a number of steps below.

\begin{claim}
The commutator $[e_x,e_d]$ is of the form specified, namely $$1+(\textrm{inner derivations})+\epsilon^3 \delta r +O(\epsilon^4 \delta)$$ for some derivation $r$.
\end{claim}
\begin{proof}
We note that any pure $\epsilon$ or pure $\delta$ terms must vanish from the commutator, leaving only terms involving both $\epsilon$ and $\delta$.  To compute the terms of size at least $\epsilon^2\delta$, we compute
\begin{align*}
&(1+\epsilon\ad(x)+\epsilon^2/2 \ad^2(x))(1+\delta d)\cdot (1-\epsilon\ad(x)+\epsilon^2/2 \ad^2(x) )(1-\delta d)\\ = &
1 + \epsilon \delta \ad(d(x)) + \epsilon^2 \delta (\ad(x*(d(x))/2)) + O(\epsilon^3\delta).
\end{align*}
So the $\epsilon\delta$ and $\epsilon^2 \delta$ terms are inner derivations.  The $\epsilon^3 \delta$ term must also be a derivation since if we write this automorphism as $1+a(\epsilon) \delta + O(\delta^2)$ it follows that $a(\epsilon)$ must be a derivation, and in particular that the $\epsilon^3$ term of $a$ is a derivation.
\end{proof}

\begin{claim}
Given $e_x$ and $d$, $r$ does not depend on the choice of $e_d$.  In particular, given a derivation $d$ of $V$ and an automorphism $e_x \in \End(V)[[\epsilon]]$, we obtain a well-defined $r\in V'$.
\end{claim}
\begin{proof}
$d$ defines $e_d$ up to $O(\delta^2)$, and therefore $[e_x,e_d]$ is defined up to $O(\delta^2)$.
\end{proof}

\begin{claim}
For fixed $e_x \in \End(V)[[\epsilon]]$, the corresponding map $d\rightarrow r$ descends to a well defined linear map $V'\rightarrow V'$.
\end{claim}
\begin{proof}
It is clear that this produces a linear map from $\der(V)\rightarrow V'$.  If $d$ is an inner derivation, $\ad(y)$ then the resulting product is $$\exp(\delta \ad(e_x(y)))\exp(-\delta \ad(y))+O(\delta^3)=1+\delta\ad(e_x(y)-y) +O(\delta^2).$$ So $r$ is an inner derivation.  Hence this descends to a map from $V'$ to $V'.$
\end{proof}

\begin{claim}
Given $x\in V$ and two possible automorphisms $e_x,e_x'\in \End(V)[[\epsilon]]$ of the form specified, $e_x$ and $e_x'$ define maps $V'\rightarrow V'$ that differ by an inner derivation of $V'$.  Hence we have a well-defined map $\rho:V\rightarrow \End(V')/\inder(V')$.
\end{claim}
\begin{proof}
Note that $e_x - e_x'$ must be of the form $\epsilon^3 f + O(\epsilon^4)$ for some derivation $f$.  It is then not hard to check that the corresponding maps $V'\rightarrow V'$ differ by the inner derivation $d\rightarrow [\ad(f),d]$.
\end{proof}

\begin{claim}\label{eChangeClaim}
Let $x\in V$,  $e_x \in \End(V)[[\epsilon]]$ as above, $y\in V[[\epsilon^{1/n}]]$ with $y=o(\epsilon)$, and $e_y = 1+\ad(y)+\ad^2(y)/2+o(\epsilon^3)$ an automorphism.  Then if we compute the above map $V'\rightarrow V'$ by taking the commutator of $e_d$ with $e_ye_x$ instead of $e_x$, we obtain the same element of $\End(V')$.
\end{claim}
\begin{proof}
This is because
$$
e_ye_xe_de_x^{-1}e_y^{-1}e_d^{-1} = (e_y(e_xe_de_x^{-1}e_d^{-1})e_y^{-1})(e_ye_de_y^{-1}e_d^{-1}).
$$
The first term above is $1+\epsilon^3\delta r + o(\epsilon^3 \delta)$ up to inner derivations, and the latter is $1+o(\epsilon^3\delta)$ up to inner derivations.
\end{proof}

\begin{claim}
The map $\rho:V\rightarrow \End(V')/\inder(V')$ is Frobenius-linear.
\end{claim}
\begin{proof}
It is clear that the map $x,d\rightarrow r$ is homogeneous of degree 3 in $x$.  We need to show that it is additive in $x$.  By the previous claim, we may substitute $e_xe_y$ for $e_{x+y}$ when computing the image of $x+y$. Additivity follows immediately from the identity
$$
e_ye_xe_de_x^{-1}e_y^{-1}e_d^{-1} = (e_y(e_xe_de_x^{-1}e_d^{-1})e_y^{-1})(e_ye_de_y^{-1}e_d^{-1}).
$$
\end{proof}

\begin{claim}\label{rhoEquationClaim}
For $x\in V$ with $\pair{x}{x}=0$, $(\rho(x))(d)=\ad(x)\ad(d(x))\ad(x)$.
\end{claim}
\begin{proof}
First we claim that we may perform our computation using $e_x$ of the form $e_x=1+\epsilon\ad(x)+\epsilon^2/2\ad^2(x)+O(\epsilon^4)$.  To show this, it suffices to check that this $e_x$ in an automorphism modulo $\epsilon^4$. This holds because
\begin{align*}
& (y+\epsilon x*y - \epsilon^2 x*(x*y))*(z+\epsilon x*z - \epsilon^2 x*(x*z))\\
= & (y+\epsilon x*y -\epsilon^2 \pair{x}{y} x)*(z+\epsilon x*z -\epsilon^2 \pair{x}{z} x)\\
= & (y*z + \epsilon((x*y)*z + y*(x*z)) \\ & + \epsilon^2 (-(x*(x*y))*z + (x*y)*(x*z) - y*(x*(x*z))) \\ & + \epsilon^3 (-\pair{x}{z}(x*y)*x - \pair{x}{y}x*(x*z))+O(\epsilon^4))\\
= & (y*z + \epsilon(x*(y*z)) -\epsilon^2(x*(x*(y*z))) + \epsilon^3(\pair{x}{z}\pair{x}{y}x-\pair{x}{y}\pair{x}{z}x) + O(\epsilon^4))\\
= & (y*z + \epsilon(x*(y*z)) -\epsilon^2(x*(x*(y*z))) + O(\epsilon^4)).
\end{align*}
Using the $e_x$ above, we get that
\begin{align*}
& e_xe_de_x^{-1}e_d^{-1}\\
= & (1+\epsilon \ad(x) - \epsilon^2\ad^2(x) +O(\epsilon^4))(1+\delta d+O(\delta^2))\\
&\cdot (1-\epsilon \ad(x) - \epsilon^2\ad^2(x) +O(\epsilon^4))(1-\delta d+O(\delta^2))\\
= & 1 + \epsilon\delta \ad(d(x)) +\epsilon^2\delta(-\ad(x*d(x)))\\
& + \epsilon^3\delta(-\ad(x)d\ad^2(x)+\ad^2(x)d\ad(x)) +o(\epsilon^3\delta).
\end{align*}
Hence,
$$
r = -\ad(x)d\ad^2(x)+\ad^2(x)d\ad(x) = \ad(x)[\ad(x),d]\ad(x) = \ad(x)\ad(d(x))\ad(x).
$$
\end{proof}

\begin{claim}
For $x\in V$ with $\pair{x}{x}=0$, the above map $\rho(x):V'\rightarrow V'$ defined by
$$
d \rightarrow \ad(x)\ad(d(x))\ad(x)
$$
is a derivation.
\end{claim}
\begin{proof}
Note that for $d,e\in V'$,
\begin{align*}
& [(\rho(x))(d),e] + [d,(\rho(x))(e)]\\
= & -\ad(e(x))\ad(d(x))\ad(x) - \ad(x)\ad(e(d(x)))\ad(x) - \ad(x)\ad(d(x))\ad(e(x))\\
& + \ad(d(x))\ad(e(x))\ad(x) + \ad(x)\ad(d(e(x)))\ad(x) + \ad(x)\ad(e(x))\ad(d(x))\\
= & \ad(x)\ad([d,e](x))\ad(x) + [\ad(d(x)),\ad(e(x))]\ad(x) - \ad(x)[\ad(d(x)),\ad(e(x))]\\
= & (\rho(x))([d,e]) + [\ad([d(x),e(x)]),\ad(x)]\\
= & (\rho(x))([d,e]) + \ad([[d(x),e(x)],x]).
\end{align*}
Which is $(\rho(x))([d,e])$ up to an inner derivation.
\end{proof}

\begin{claim}
The map $\rho:V\rightarrow \End(V')/\inder(V')$ gives a Frobenius linear map $V\rightarrow V''$.
\end{claim}
\begin{proof}
We already have that the map is Frobenius-linear.  Furthermore the above claim implies that the image is a derivation of $V'$ as long as $\pair{x}{x}=0$.  Since such vectors span $V$ we have by linearity that the image lies entirely in $\der(V)$.  Hence we have a map $V\rightarrow \der(V')/\inder(V') = V''$.
\end{proof}

\begin{claim}
The map $\rho:V\rightarrow V''$ is a map of Lie Algebras.
\end{claim}
\begin{proof}
We need to show that $\rho(x*y) = \rho(x)*\rho(y)$.  We note that by Claim \ref{eChangeClaim} that we may substitute $[e_x,e_y]$ (with $\epsilon$ replaced by $\epsilon^{1/2}$) for $e_{x*y}$ in our computation of $\rho(x*y)$.  Equivalently, we map compute $(\rho(x*y))(d)$ as the $\epsilon^6\delta$ term of $[[e_x,e_y],e_d]$.  We note that
\begin{align*}
[[e_x,e_y],e_d]=e_y^{-1}e_x^{-1}e_ye_xe_de_x^{-1}e_y^{-1}e_xe_ye_d^{-1}
\end{align*}
Is conjugate to
$$
e_ye_xe_de_x^{-1}e_y^{-1}e_xe_ye_d^{-1}e_y^{-1}e_x^{-1}.
$$
This conjugation by $e_x e_y$ should not effect our final output since it should send inner derivations to inner derivations and not modify our $\epsilon^6\delta$ term by more than $O(\epsilon^7\delta)$.  The above equals
$$
e_{y(x(d))}e_{y(d)}e_{x(d)}e_de_{x(y(d))}^{-1}e_{x(d)}^{-1}e_{y(d)}^{-1}e_d^{-1}
$$
where
\begin{align*}
e_{x(d)} & = e_xe_de_x^{-1}e_d^{-1}\\
e_{y(d)} & = e_ye_de_y^{-1}e_d^{-1}\\
e_{x(y(d))} & = e_xe_{y(d)}e_x^{-1}e_{y(d)}^{-1}\\
e_{y(x(d))} & = e_ye_{x(d)}e_y^{-1}e_{x(d)}^{-1}.
\end{align*}
Since these are each $1+O(\delta)$, they commute modulo $\delta^2$, and hence modulo $\delta^2$, the above is
$$
e_{y(x(d))}e_{x(y(d))}^{-1}.
$$
Up to inner derivations this is
$$
1+\epsilon^6\delta( (\rho_V(y))((\rho_V(x))(d)) - (\rho_V(x))((\rho_V(y))(d)))+o(\epsilon^6\delta).
$$
Hence $\rho_V(x*y) = \rho_V(x)*\rho_V(y)$ as desired.
\end{proof}

As in the previous case, if we have a $V\in\Oc$ and a $\frobenius^m$-linear map $F:V'\rightarrow V$, we can define $\FR:V\rightarrow V$ a $\frobenius^{2m+1}$-linear map thus giving $V$ a $\F_{3^{2m+1}}$-structure.  We can then construct an endomorphism $\sigma$ of $G_2(V)=\Aut(V)$ by $\fr^{-m}(\sigma(g)) \circ F = F \circ g'$, so that $\sigma^2=\FR$.

\subsection{The Deligne-Lusztig Curve}

Before we can construct the Deligne-Lusztig curve we need to understand the Borel subgroups of $G$.  Given $V\in \Oc$ we can consider the algebraic group $G=\Aut(V)$.  We consider $B$ a Borel subgroup of $G$.  $B$ is determined by a line $L=\langle x \rangle$ and a plane $M=\langle x,y\rangle\supset L$ that are fixed by it.  These have the property that both $\pair{-}{-}$ and $*$ are trivial on both $L$ and $M$.  So $\pair{x}{x} = \pair{x}{y} = \pair{y}{y} = 0, x*y = 0$.  $B$ will also fix a 3 dimensional space $S=\langle x,y,z \rangle = \ker(\ad(x)) \supset M$.  Given a Borel $B$ of $G$, there should be a corresponding Borel $B'$ of $G'=\Aut(V')$ by applying $'$ to each element of $B$.  Letting $W=\ker(*:\Lambda^2 V\rightarrow V)$, we recall that $V'=W/W^\perp$.  We note that if $g\in B$ that $g'$ must fix the line containing $x\wedge y\in W$.  Note that $\pair{x\wedge y}{x\wedge y} = 0$.  Furthermore $x\wedge y$ cannot be an inner automorphism because $\ad(a)$ has rank 6 if $\pair{a}{a}\neq 0$ and rank 4 if $\pair{a}{a}=0$ while $x\wedge y$ has rank 2.  Therefore $g'$ fixes the line generated by $x\wedge y$ in $V'$.  $B'$ also fixes the plane $\langle x\wedge y, x\wedge z\rangle$.  This is of the type described because $(x\wedge y)*(x\wedge z)(a)=0$ since $(x\wedge z)(a) \in \langle x,z\rangle$ and both $x$ and $y$ are perpendicular to $x$ and $z$.  This means that if $B$ is the Borel fixing $L$ and $M$, then $\sigma(B)$ is the Borel fixing $F(\langle x\wedge y\rangle)$ and $F(\langle x\wedge y, x\wedge z\rangle)$.

The Deligne-Lusztig curve $C$ is the set of Borels $B$ so that $B$ and $\sigma(B)$ fix the same line.  We can produce an explicit embedding $C\hookrightarrow \mathbb{P}(\Lambda^2(V))$ by sending $B$ to the line of $\omega = x\wedge y$.  We claim that this embedding gives the curve defined by the following relations:
\begin{itemize}
\item The linear relations that tell us that $*(\omega)=0$.
\item The relations that tell us that $\omega$ is a rank 2 tensor, and hence a pure wedge of two elements of $V$.  Note that these are the relations used to define the Plucker embedding of a Grassmannian.
\item The relation $\pair{\omega}{\omega}=0$
\item The relations that tell us that $\langle \omega,F(\omega)\wedge V\rangle$ has dimension at most 6.
\end{itemize}
These are clearly satisfied for points in the image of our embedding.  They also cut them out set-theoretically.  The first and second relations guarantee that $\omega$ represents a non-zero element of $V'$.  Let $x=F(\omega)\neq 0$.  The fourth relation guarantees that $\omega = x\wedge y$ for some $y$.  The first relation says that $x*y=0$.  The third relation implies that $\pair{x}{x}=0$.  Hence we have the Borel $B$ that fixes $\langle x \rangle$ and $\langle x,y\rangle$ the unique Borel subgroup so that $x\wedge y$ is parallel to $\omega$.  This embedding is smooth because the coordinates of the flag variety can all be written as polynomials in $\omega$ for points on $C$.

These equations also cut out the curve scheme-theoretically.  We show this for $m>0$ by showing that the tangent space to the scheme defined by these equations is one dimensional at every $\bar{\F}$-point.  Suppose that we are at some point on $C$ with projective coordinate $\omega$.  We wish to consider the space of possible vectors $d\omega$ in the tangent space.  Note that $dF(\omega)=0$ and that therefore, $d\omega$ must lie in the six dimensional space defined by $F(\omega)\wedge V$.  We must also have $*(d\omega)=0$ so $d\omega$ must lie in $F(\omega) \wedge (\ker(\ad(F(\omega))))$.  Since $\pair{F(\omega)}{F(\omega)}=0$, $(\ker(\ad(F(\omega))))$ is three dimensional.  Since $(\ker(\ad(F(\omega))))$ contains $F(\omega)$, $F(\omega) \wedge (\ker(\ad(F(\omega))))$ is two dimensional.  Finally when we project down to the tangent space to $\mathbb{P}(\Lambda^2 V)$ at this point (which is $\Lambda^2(V)/\langle \omega \rangle$), we are left with a one-dimensional tangent space.

\subsection{Divisors of Note}

We begin with some preliminaries.  Let $C$ be the Deligne-Lusztig curve corresponding to $\ree$ defined over the field $\F_{3^{2m+1}}$.  Let $q_0=3^m$.  Let $q=3q_0^2=3^{2m+1}$.  Let $q_- = q-3q_0+1$ and $q_+ = q+3q_0+1$.  Note that $q_-q_+ = q^2-q+1$.  We let $G=G_2$ with an endomorphism $\sigma$.  We note that $C$ admits a natural $G^\sigma$ action.  We will be considering the projective model of $C$ described above with parameter $\omega$.

We first want to compute the degree of this embedding.  We do this by finding a polynomial that vanishes exactly at the $\F_q$ points.  We will think of $F$ as a $\frobenius^m$-linear map $V'\rightarrow V$.  If $B$ is the Borel fixing the line, plane and space $L=\langle x\rangle, P=\langle x,y\rangle, S=\langle x,y,z\rangle$, $\sigma(B)$ fixes the line and plane $\langle F(x\wedge y)\rangle, \langle F(x\wedge y),F(x\wedge z)\rangle$.  For $B$ to correspond to a point on $C$, then we must have that $F(x\wedge y)$ is a multiple of $x$.  $B$ is defined over $\F_q$ if and only if $B=\sigma(\sigma(B))$.  We claim that this happens if and only if $B=\sigma(B)$.  The if direction is clear.  For only if, note that if $B=\sigma(\sigma(B))$, $B$, $\sigma(B)$ and $\sigma(\sigma(B))$ must all fix the same line. Therefore $F(x\wedge y)$ is a multiple of $F(F(x\wedge y)\wedge F(x\wedge z))$, which is a multiple of $F(x\wedge F(x\wedge z))$.  This happens if and only if $F(x\wedge z)$ is in the span of $x$ and $y$, which in turn implies that $B=\sigma(B)$.  So in fact, $B\in C$ is an $\F_q$-point if and only if $F(x\wedge z)$ is in the plane of $x$ and $y$.

Notice that since $(x\wedge y)*(x\wedge z)=0$ that $0=F(x\wedge y)*F(x\wedge z) = x*F(x\wedge z)$.  Therefore $F(x\wedge z)$ is in the span of $x,y,z$.  Therefore we have that $(F(x\wedge z)\wedge x)\wedge(x\wedge y)$ is always a multiple of $(x\wedge y)\wedge(x\wedge z)$ in $V''$.  Furthermore this multiple is 0 if and only if the point is defined over $\F_q$.  We next claim that $\rho_V(x)$ is a (non-zero) multiple of $(x\wedge y)\wedge(x\wedge z)$.  This is equivalent to saying that $(\rho_V(x))(V')$ is in the span of $x\wedge y$ and $x\wedge z$.  This in turn is equivalent to saying that the image $((\rho_V(x))(V'))(V)$ is contained in the span of $x,y$ and $z$.  But note that $\langle x,y,z\rangle = \ker(\ad(x))$.  Furthermore, by claim \ref{rhoEquationClaim}, $$x*((\rho_V(x))(d))(a) = x*(x*(d(x)*(x*a))) = \pair{x}{d(x)*(x*a)}x.$$  We need to show that $\pair{x}{d(x)*(x*a)}=0$. This is the cubic form applied to $x,d(x),(x*a)$.  Hence this is $-\pair{d(x)}{x*(x*a)} = -\pair{d(x)}{\pair{x}{a}x} = -\pair{x}{a}\pair{d(x)}{x} = 0$ since $d$ must be an anti-symmetric linear operator.  Hence $\rho_V(x)$ is a multiple of $(x\wedge y)\wedge (x\wedge z)$.

We now consider the polynomial
\begin{equation}\label{reeFqPoly}
\frac{[(F(\omega)\wedge(F(F(\omega)\wedge z)))\wedge \omega]\otimes [\rho_V(F(\omega))]^{\otimes (q_0-1)}}{[\omega\wedge(F(\omega)\wedge z)]^{\otimes q_0}}.
\end{equation}
Both numerator and denominator are elements of the $q_0^{th}$ tensor power of the space of multiples of $(x\wedge y)\wedge(x\wedge z)$.  Note that each of numerator and denominator are proportional to the $q_0^{th}$ power of $z$ as an element of $S/M$.  Therefore the number produced is independent of the choice of $z$.  The $\omega$-degree of the above is:
$$
[q_0+q_0^2+1] + (q_0-1)[3q_0] - q_0[q_0+1] = q_0+q_0^2+1+q-3q_0-q_0^2-q_0 = q - 3q_0+1.
$$
The polynomial in Equation \ref{reeFqPoly} clearly vanishes exactly when $F(x\wedge z)$ is in $M$, or on points defined over $\F_q$.  We claim that it vanishes to degree 1 on such points (at least for $m>0$).  Consider an analytic coordinate around such a point.  Since $m>0$ the derivative of $F(\omega)$ is 0.  Hence we can consider the above with $z$ so that $dz=0$.  We wish to show that the derivative of the above polynomial is non-zero at such a point.  To do so it suffices to show that the derivative of $(F(\omega)\wedge(F(F(\omega)\wedge z)))\wedge \omega$ is non-zero.  But it should be noted that the derivative of $(F(\omega)\wedge(F(F(\omega)\wedge z)))$ is 0.  Hence as $\omega$ varies, the wedge has non-zero derivative.  This completes the proof.  Therefore this polynomial of degree $q_-$ defines a divisor of degree $q^3+1$.  Hence our embedding must be of degree $q_+(q+1).$

We next consider the divisor defined by the polynomial $\pair{\omega}{\FR(\omega)}$.  This defines a divisor of degree $(q+1)^2q_+$.  This divisor is clearly $G^{\sigma}$-invariant.  On the other hand the degree is too small to contain any orbits except for the orbit of $\F_q$-points.  Unfortunately, the size of this orbit does not divide the degree of this divisor.  Therefore $\pair{\omega}{\FR(\omega)}$ must vanish on $C$.  Similarly $\pair{\omega}{\FR^2(\omega)}$ must vanish on $C$.

Consider the polynomial $\pair{\omega}{\FR^3(\omega)}$.  We claim that this vanishes to degree exactly $q_+(q+1)$ on the $\F_q$-points of $C$ and nowhere else.  For the former, consider an $\F_q$-point $\omega=\omega_0$.  Then near this point $\pair{\omega}{\FR^3(\omega)}$ agrees with $\pair{\omega}{\omega_0}$ to degree $q^3$.  Since the latter cannot vanish to degree more than $q_+(q+1)$, being a degree 1 polynomial, this means that $\pair{\omega}{\FR^3(\omega)}$ cannot vanish identically.  Since it is $G^{\sigma}$-invariant, but of too small a degree to contain any orbit but that of the $\F_q$-points, it must vanish to degree $n$ on each $\F_q$-point for some $n$ and nowhere else.  Comparing degrees yields $n=q_+(q+1)$.  Note also that this implies that the polynomial $\pair{\omega}{\omega_0}$ vanishes at the point defined by $\omega_0$ with multiplicity $q_+(q+1)$ and nowhere else.

Next consider the polynomial $\pair{\omega}{\FR^4(\omega)}$.  This vanishes on the $\F_{q^6}$-points of $C$.  This is because for such points, $$\FR^2(\pair{\omega}{\FR^4(\omega)})=\pair{\FR^2(\omega)}{\FR^6(\omega)} = \pair{\FR^2(\omega)}{\omega} = 0.$$  Furthermore this polynomial agrees with $\pair{\omega}{\FR^3(\omega)}$ to order $q^3$ on the $\F_q$-points.  Therefore this polynomial vanishes to degree 1 on the $\F_{q^6}$-points and degree $q_+(q+1)$ on the $\F_q$-points.  By degree counting, it vanishes nowhere else.

Lastly consider the polynomial $\pair{\omega}{\FR^5(\omega)}$.  Analogously to the above this vanishes on the $\F_{q^7}$-points and to degree exactly $q_+(q+1)$ on the $\F_q$-points.  The remaining degree is
\begin{align*}
& q_+(q+1)(q^5+1) - q^3(q+1)q_+q_-(q-1) - (q^3+1)q_+(q+1) \\ = & q_+(q+1)(q^5+1-q^3-1-q^5+3q_0q^4-3q_0q^3+q^3)\\  = & q_+(q+1)(3q_0q^4-3q_0q^3).
\end{align*}
Since the remainder of this divisor is $G^\sigma$-invariant and the only orbit small enough to fit is the orbit of $\F_{q^6}$-points, this polynomial must vanish to degree 1 on the $\F_{q^7}$-points, to degree $3q_0$ on the $\F_{q^6}$-points, and to degree $q_+(q+1)$ on the $\F_q$-points.

We have several polynomials of interest.  $F_3:=\pair{\omega}{\FR^3(\omega)}, F_4:=\pair{\omega}{\FR^4(\omega)}, F_5:=\pair{\omega}{\FR^5(\omega)}$.  We also have the polynomials that vanish exactly at the $\F_{q^n}$-points.  $P_1:= F_3^{\frac{1}{q_+(q+1)}}$, $P_6:= \frac{F_4}{F_3}$, $P_7:=\frac{F_5}{F_3P_6^{3q_0}}$.  We note that these are all polynomials ($P_1$ as above is defined only up to a $q_+(q+1)^{st}$ root of unity but could be taken, for example, to be the Polynomial in Equation \ref{reeFqPoly}).  We claim that if the correct root is taken in the choice of $P_1$, then we have that
\begin{equation}\label{reeCurveRelation}
P_1^{q^3(q-1)}+P_6^{q_-}-P_7=0.
\end{equation}
Or unequivocally,
$$
F_3^{q^3(q-1)} = (P_6^{q_-}-P_7)^{q_+(q+1)}.
$$

We prove Equation \ref{reeCurveRelation} by showing that for proper choice of $P_1$, that this polynomial vanishes at all of the $\F_{q^7}$-points and all of the $\F_q$-points.  Since this will be more points than the degree of the polynomial would allow, the polynomial must vanish on $C$.

For the $\F_{q^7}$-points, it will suffice to show that $F_3^{q^3(q-1)}(Q) = P_6^{q_-q_+(q+1)}(Q) = \left(\frac{F_4(Q)}{F_3(Q)}\right)^{q^3+1}$ for $Q$ a $\F_{q^7}$-point. For this it suffices to show that $F_3^{q^4+1}(Q)=F_4^{q^3+1}(Q)$.  But we have that
\begin{align*}
F_3^{q^4+1}(Q) & = \pair{\omega}{\FR^3(\omega)}^{q^4}\pair{\omega}{\FR^3(\omega)}\\
& = \pair{\FR^4(\omega)}{\FR^7(\omega)}\pair{\omega}{\FR^3(\omega)}\\
& = \pair{\FR^4(\omega)}{\omega}\pair{\FR^7(\omega)}{\FR^3(\omega)}\\
& = \pair{\omega}{\FR^4(\omega)}^{q^3}\pair{\omega}{\FR^4(\omega)}\\
& = F_4^{q^3+1}(Q).
\end{align*}

For the $\F_q$-points we show that $P_6^{q_-}(Q)=P_7(Q)=\frac{F_5(Q)}{F_3(Q)P_6^{3q_0}(Q)}$ for $Q$ and $\F_q$-point of $C$.  This is equivalent to showing that
$$
\left(\frac{F_4(Q)}{F_3(Q)} \right)^{q+1} = P_6^{q+1}(Q) = \frac{F_5(Q)}{F_3(Q)}.
$$
To do this we will show that the polynomials
$$
F_4^{q+1}
$$
and
$$
F_5F_3^q
$$
agree to order more than $q_+(q+1)^2$.  But this follows immediately from the fact that each of these polynomials is a product of $q+1$ of the $F_i$ which are in turn polynomials that
\begin{itemize}
\item Vanish to order $q_+(q+1)$ at $Q$
\item Agree with $F_3$ to order $q^3$ at $Q$
\end{itemize}
This completes our proof of Equation \ref{reeCurveRelation}.

\section{Acknowledgements}

This work was done with the support of an NDSEG graduate fellowship.

\end{document}